\documentclass[11pt]{amsart}
\usepackage{amsfonts,amssymb,amscd,amsmath,enumerate,verbatim,calc,amsthm}
\input xy
\xyoption{all}

\headsep=1cm \numberwithin{equation}{section}
\newtheorem{thm}{Theorem}[section]
\newtheorem{cor}[thm]{Corollary}
\newtheorem{lem}[thm]{Lemma}
\newtheorem{prop}[thm]{Proposition}
\newtheorem{defn}[thm]{Definition}

\newtheorem{exam}[thm]{Example}
\newtheorem{rem}[thm]{Remark}

\DeclareMathOperator{\Supp}{Supp}
\DeclareMathOperator{\Spec}{Spec}

\newcommand{\gr}{\mbox{grade}\,}
\def\id{\operatorname{\mathsf{id}}}

\newcommand{\vdim}{\mbox{vdim}\,}

\def\cd{\operatorname{\mathsf{cd}}}

\newcommand{\h}{\mbox{ht}\,}

\newcommand{\E}{\mbox{E}}

\renewcommand{\H}{\mbox{H}}
\newcommand{\V}{\mbox{V}}

\newcommand{\fa}{\mathfrak{a}}

\newcommand{\fm}{\mathfrak{m}}
\newcommand{\fp}{\mathfrak{p}}

\def\depth{\operatorname{\mathsf{depth}}}

\def\Hom{\operatorname{\mathsf{Hom}}}
\def\dim{\operatorname{\mathsf{dim}}}

\def\Ext{\operatorname{\mathsf{Ext}}}

\begin{document}
\title[On the relative cohen-macaulay modules ]
 { On the relative cohen-macaulay modules}

\bibliographystyle{99}

     \author[M. R. Zargar]{Majid Rahro Zargar }

\address{Majid Rahro Zargar\\ School of Mathematics, Institute for Research in Fundamental Sciences (IPM), P.O. Box: 19395-5746, Tehran, Iran.}
\email{zargar9077@gmail.com}

\keywords{Local cohomology, Gorenstein module, relative Cohen-Macaulay.\\
This research was in part supported by a grant from IPM (No. 92130024)}


\begin{abstract}Let $R$ be a commutative Noetherian local ring and let $\fa$ be a proper ideal of $R$. In this paper, as a main result, it is shown that if $M$ is a Gorenstein $R$-module with $c=\h_{M}\fa$, then $\H_{\fa}^{i}(M)=0$ for all $i\neq c$ is completely encoded in homological properties of $\H_{\fa}^{c}(M)$, in particular in its Bass numbers. Notice that, this result provides a generalization of a result of Hellus and Schenzel which has been proved before, as a main result, in the case where $M=R$.

\end{abstract}

\maketitle

\section{introduction}Throughout this paper, $R$ is a commutative Noetherian ring, $\fa$ is a proper ideal of $R$ and $M$ is an $R$-module. For a prime ideal $\fp$ of $R$, the residue class field $R_{\fp}/\fp R_{\fp}$ is denoted by $k(\fp)$. For each non-negative integer $i$, let $\H_{\fa}^{i}(M)$ denotes the $i$-th local cohomology module of $M$ with respect to $\fa$; see \cite{BRS} for its definition and basic results. A Gorenstein module over a local ring $R$ is a maximal
Cohen-Macaulay module of finite injective dimension. This concept was introduced by Sharp in \cite{RSH} and studied extensively by him and other authors. In the present paper, we will use the concept of relative Cohen-Macaulay modules which is a generalization of the concept of Cohen-Macaulay modules. This kind of modules has been studied in \cite{HS1} under the title of cohomologically complete intersections and continued in \cite{MZ}. In \cite{HS1} Hellus and Schenzel, as a main result, showed that if $(R,\fm, k)$ is a Gorenstein local ring and $\fa$ is an ideal of $R$ with {$\h_{R}\fa=c$} and {$\dim R/ \fa =d$} such that $R$ is relative Cohen-Macaulay in {$\V(\fa)\setminus\{\fm\}$}, then the following statements are equivalent.
\begin{itemize}
\item[(i)]{{$\H^{i}_\fa(R)=0$ for all $i\neq c$, i.e $R$ is relative Cohen-Macaulay with respect to $\fa$.}}
\item[(ii)]{{$\H^d_{\fm}(\H^{c}_{\fa}({R}))\cong \E_{R}(k)$ and $\H^i_{\fm}(\H^{c}_{\fa}({R}))=0$ for all $i\neq d$. }}
\item[(iii)]{{$\Ext_{R}^{d}(k,\H^{c}_{\fa}(R))\cong k$ and $\Ext_{R}^{i}(k,\H^{c}_{\fa}(R))=0$ for all $i\neq d$}.}
\item[(iv)]{{$\mu^i(\fm,\H_{\fa}^c(R))=\delta_{di}$}}.
\end{itemize}
{Moreover, if $\fa$ satisfies the above conditions, it follows that $\hat{R}^{\fa}\cong\Hom_{R}(\H^{c}_\fa(R),\H^{c}_\fa(R))$ and $\Ext^{i}_{R}(\H^{c}_\fa(R),\H^{c}_\fa(R))=0$ for all $i\neq0$, where $\hat{R}^{\fa}$ denotes the $\fa$-adic completion of $R$.}

As a main result, in Theorem 2.9, we generalize the above result for a Gorenstein $R$-module $M$. Indeed, it is shown that if $M$ is a Gorenstein $R$-module with {$c =\h_{M}\fa$} and {$d=\dim M/\fa M$} such that $M$ is relative Cohen-Macaulay in {$\Supp_{R}(M/\fa M)\setminus\{\fm\}$}, then the following statements are equivalent.
 \begin{itemize}
\item[(i)]{{$\H^{i}_\fa(M)=0$ for all $i\neq c$, i.e $M$ is relative Cohen-Macaulay with respect to $\fa$.}
\item[(ii)]{$\H^d_{\fm}(\H^{c}_{\fa}({M}))\cong \E_{R}(k)^{\tiny{r(M)}}$ and $\H^i_{\fm}(\H^{c}_{\fa}({M}))=0$ for all $i\neq d$. }
\item[(iii)]{$\Ext_{R}^{d}(k,\H^{c}_{\fa}(M))\cong k^{\tiny{r(M)}}$ and $\Ext_{R}^{i}(k,\H^{c}_{\fa}(M))=0$ for all $i\neq d$}.
\item[(iv)]{$\mu^i(\fm, \H_{\fa}^c(M))={\tiny{r(M)}}\delta_{di}$}.}
\end{itemize}
{Moreover, if $\fa$ satisfies the above conditions, it follows that $\Hom_{R}(M,M)\otimes_{R}\hat{R}^{\fa}\cong\Hom_{R}(\H^{c}_\fa(M),\H^{c}_\fa(M))$ and $\Ext^{i}_{R}(\H^{c}_\fa(M),\H^{c}_\fa(M))=0$ for all $i\neq0$, where $r(M)$ denotes the type of $M$.}
\section{Main results}
\begin{defn} \emph{We say that a finitely generated $R$-module $M$ is \textit{relative Cohen–
Macaulay with respect to $\fa$} if there is precisely one non-vanishing local cohomology module of $M$ with
respect to $\fa$. Clearly this is the case if and only if $\gr(\fa,M)=\cd(\fa,M)$, where $\cd(\fa,M)$ denotes the cohomological dimension of $M$ with respect to $\fa$, which is the largest integer $i$ for which $\H_{\fa}^i(M)\neq0$. Observe that the notion of relative Cohen-Macaulay module is connected with the notion of cohomologically complete intersection ideal which has been studied in \cite{HS1}.
}
\end{defn}
\begin{rem}\emph{Let $M$ be a relative Cohen-Macaulay module with respect to $\fa$ and let $\cd(\fa,M)=n$. Then, in view of \cite[Theorems 6.1.4, 4.2.1, 4.3.2]{BRS}, it is easy to see that $\Supp\H^{n}_{\fa}(M)=\Supp({M}/{\fa M})$ and $\h_{M}\fa=\gr(\fa,M)$, where
$\h_{M}\fa =\inf\{\ \dim_{R_{\fp}}M_{\fp} |~ \fp\in\Supp(M/\fa M) ~\}$.}
\end{rem}
\begin{defn}\emph{( See \cite[Theorem 3.11]{RSH}.) A non-zero finitely generated module $M$ over a local ring $(R,\fm)$ is said to be a Gorenstein $R$-module if the following equalities hold true $$\depth M=\dim M=\id_{R}M=\dim R.$$}
\end{defn}
\begin{defn} \emph{For any prime ideal $\fp$ of $R$, the $i$-th Bass number of $M$ with respect to $\fp$ is defined by $\mu^{i}(\fp, M):=\vdim_{k{(\fp)}}\Ext_{R_{\fp}}^{i}(k(\fp),M_{\fp})$. If  $(R,\fm)$ is local and $M$ is finitely generated of depth $t$, then the number {$\mu^{t}(\fm,M)$} is called type of $M$ and is denoted by $r(M)$.}
\end{defn}

Let $M$ be a non-zero finitely generated module over a local ring $(R,\fm)$ and let $\fa$ be an ideal of $R$. Let $\E_{R}^{\textbf{.}}(M)$
be a minimal injective resolution for $M$. It is a well-known fact that
$\E_{R}^{\textbf{.}}(M)^i=\bigoplus_{\tiny{\fp\in\Spec(R)}} \mu^{i}(\fp,M)\E_{R}(R/\fp)$,
where $\E_{R}(R/\fp)$ denotes the injective hull of $R/\fp$.
Now, let $c=\gr(\fa,M)$. Then, $\Gamma_{\fa}(\E_{R}^{\textbf{.}}(M)^i)=0$
for all $i<c$.
Therefore $\H_{\fa}^c(M)=\ker(\Gamma_{\fa}(\E_{R}^{\textbf{.}}(M))^c\longrightarrow\Gamma_{\fa}(\E_{R}^{\textbf{.}}(M))^{c+1})$.
This observation provides an embedding $0\longrightarrow \H_{\fa}^c(M)[-c]\longrightarrow\Gamma_{\fa}(\E_{R}^{\textbf{.}}(M))$ of complexes of $R$-modules where $\H_{\fa}^c(M)[-c]$ is considered as a complex concentrated in homological degree zero.

\begin{defn}\emph{( See \cite[Definition 4.1]{HS}.) The cokernel of the embedding {$$0\longrightarrow\H_{\fa}^c(M)[-c]\longrightarrow\Gamma_{\fa}(\E_{R}^{\textbf{.}}(M))$$} is denoted by $C^{\textbf{.}}_{M}(\fa)$ and is called the truncation complex. Therefore, there is a short exact sequence  {$$0\longrightarrow \H_{\fa}^c(M)[-c]\longrightarrow\Gamma_{\fa}(\E_{R}^{\textbf{.}}(M))\longrightarrow C^{\textbf{.}}_{M}(\fa)\longrightarrow 0$$}
of complexes of $R$-modules. We observe that {$\H^{i}( C^{\textbf{.}}_{M}(\fa))\cong \H_{\fa}^i(M)$} for all $i>c$ while {$\H^{i}( C^{\textbf{.}}_{M}(\fa))=0$ for all $i\leq c$.}}
\end{defn}
The following proposition is of assistance in the proof of the main result.
\begin{prop} Assume that $M$ is a Gorenstein module over a local ring $R$. Then, with the previous notation the following statements are true.
  \begin{itemize}
\item[(i)]\emph{{There is an exact sequence $$0\rightarrow\Ext^{0}_{R}(C^{\textbf{.}}_{M}(\fa),M)\rightarrow \Hom_{R}(M,M)\otimes_{R}\hat{R}^{\fa}\rightarrow\Ext_{R}^{c}(\H^{c}_{\fa}(M),M)\rightarrow\Ext^{1}_{R}(C^{\textbf{.}}_{M}(\fa),M)\rightarrow 0.$$}}
\item[(ii)]\emph{{There are isomorphisms $\Ext_{R}^{i+c}(\H^{c}_{\fa}(M),M)\cong\Ext^{i+1}_{R}(C^{\textbf{.}}_{M}(\fa),M)$ for all $i>0$}.}
\item[(iii)]\emph{{Suppose that $M$ is relative Cohen-Macaulay with respect to $\fa$. Then  $\Ext_{R}^{c}(\H^{c}_{\fa}(M),M)\cong\Hom_{R}(M,M)\otimes_{R}\hat{R}^{\fa}$ and $\Ext_{R}^{i+c}(\H^{c}_{\fa}(M),M)=0$ for all $i\neq0$}.}
\end{itemize}

\begin{proof}We first notice that, since $M$ is Gorenstein, $R$ is Cohen-Macaulay and $\hat M$ is a Gorenstein $\hat{R}$-module. Now, let $\omega_{\hat{R}}$ be a canonical module of the ring $\hat{R}$. Then, in view of \cite[Exercise 3.3.28]{BH}, $\hat{M}$ is isomorphic to a direct sum of finitely many copies of $\omega_{\hat{R}}$. Hence, one can use \cite[Theorem 3.3.10(c)]{BH}, \cite[Theorem 3.3.4(d)]{BH} and the fact that $\Ext_{\hat{R}}^{i}(\hat{M},\hat{M})\cong\Ext_{R}^{i}(M,M)\otimes_{R}\hat{R}$ for all $i$, to see that $\Ext_{R}^{i}(M,M)=0$ for all $i>0$ and that $\Hom_{R}(M,M)$ is a flat $R$-module.

(i): Let $E^{\textbf{.}}$ be a minimal injective resolution for $M$. Then, since $E^{\textbf{.}}$ is a bounded complex of injective $R$-modules, by applying the functor $\Hom_{R}(-,E^{\textbf{.}})$ on the exact sequence of $R$-complexes in 2.5 we obtain the following short exact
sequence of complexes
$$0\longrightarrow\Hom_{R}({C^{\textbf{.}}_{M}(\fa)},E^{\textbf{.}})\longrightarrow\Hom_{R}(\Gamma_{\fa}(E^{\textbf{.}}),E^{\textbf{.}})\longrightarrow\Hom_{R}(\H_{\fa}^c(M),E^{\textbf{.}})[c]\longrightarrow0.$$
Now, let $\b{x}= x_{1},\ldots ,x_{n}$ be a generating set of the ideal $\fa$, and let $\check{C}_{\b{x}}(R)$ be the $\check{C}$ech complex of $R$ with respect to $\b{x}$. Then, in view of \cite[Theorem 1.1]{SE}, there exists an isomorphism $\textbf{R}\Gamma_{\fa}(M)\simeq\check{C}_{\b{x}}(R)\otimes_R^{\textbf{L}}M$ in the derived category. Now, since $\Ext_{R}^{i}(M,M)=0$ for all $i>0$, we see that $\Hom_{R}(M,M)\simeq\textbf{R}\Hom_{R}(M,M)$. Thus, we get the following isomorphisms
{ \[\begin{array}{rl}
 \textbf{R}\Hom_{R}(\textbf{R}\Gamma_{\fa}(M),M)&\simeq\textbf{R}\Hom_{R}((\check{C}_{\b{x}}(R)\otimes_R^{\textbf{L}}M),M)\\
 &\simeq\textbf{R}\Hom_{R}(\check{C}_{\b{x}}(R),\textbf{R}\Hom_{R}(M,M))\\
 &\simeq\textbf{R}\Hom_{R}(\check{C}_{\b{x}}(R),\Hom_{R}(M,M)),
 \end{array}\]}
in the derived category. Now, since $\Hom_{R}(M,M)$ is flat, one can use \cite[Theorem 1.1]{SE} and the above isomorphisms to deduce that $\H^{i}(\textbf{R}\Hom_{R}(\textbf{R}\Gamma_{\fa}(M),M)=0$ for all $i\neq0$ and that $\H^{0}(\Hom_{R}(\Gamma_{\fa}(E^{\textbf{.}}),E^{\textbf{.}}))\cong\Hom_{R}(M,M)\otimes_{R}\hat{R}^{\fa}$. On the other hand, since $\Gamma_{\fa}((E^{\textbf{.}})^i)=0$ for all $i<c=\gr(\fa,M)$ and $\Hom_{R}(N,X)=\Hom_{R}(N,\Gamma_{\fa}(X))$ for any $\fa$-torsion $R$-module N and for all $R$-modules X, one can deduce that $\Ext_{R}^{i}(\H_{\fa}^{c}(M),M)=0$ for all $i<c$. Therefore, with the aid of the above considerations, the induced long exact
cohomology sequence of the above exact sequence of complexes provides the statements (i) and (ii) of the claim.

(iii): Assume that $M$ is relative Cohen-Macaulay with respect to $\fa$. Then, one can easily check that the complex $C^{\textbf{.}}_{M}(\fa)$ is exact; and so the complex $\Hom_{R}({C^{\textbf{.}}_{M}(\fa)},E^{\textbf{.}})$ is also exact.
Therefore, $\Ext^{i}_{R}(C^{\textbf{.}}_{M}(\fa),M)=\H^{i}(\Hom_{R}({C^{\textbf{.}}_{M}(\fa)},E^{\textbf{.}}))=0$ for all $i$. Hence, the assertion follows from (i) and (ii).
\end{proof}
\end{prop}

The following lemma and definition are needed in the proof of the next result.
\begin{lem}\emph{( See \cite[Proposition 4.1]{HS1}.)} Let $n$ be a non-negative integer and let $M$ be an arbitrarily $R$-module over a local ring $(R,\fm)$. Then the following conditions are equivalent.
 \begin{itemize}
\item[(i)]\emph{{$\H^{i}_{\fm}(M)=0$ for all $i<n$.}
\item[(ii)]{$\Ext_{R}^i(R/\fm,M)=0$ for all $i<n$}.}
\end{itemize}
\end{lem}
\begin{defn}\emph{Let $(R,\fm)$ be a local ring and let $\fa$ be a proper ideal of $R$. Then,
we say that a non-zero $R$-module $M$ is relative Cohen-Macaulay in $\Supp_{R}(M/\fa M)\setminus\{\fm\}$, whenever $M_{\fp}$ is relative Cohen-Macaulay with respect to $\fa R_{\fp}$ and $\cd(\fa R_{\fp}, M_{\fp})=c$ for all $\fp\in\Supp_{R}(M/\fa M)\setminus\{\fm\}$, where $c=\h_{M}\fa$.}
\end{defn}
The following theorem extends the main result \cite[Theorem 0.1]{HS1}.
\begin{thm}Let $\fa$ be an ideal of a local ring $(R,\fm,k)$ and $M$ a Gorenstein $R$-module. Suppose that $M$ is relative Cohen-Macaulay in \emph{$\Supp_{R}(M/\fa M)\setminus\{\fm\}$}. Set \emph{$c :=\h_{M}\fa$} and \emph{$d:=\dim M/\fa M$}. Then the following statements are equivalent.
 \begin{itemize}
\item[(i)]\emph{{$\H^{i}_\fa(M)=0$ for all $i\neq c$, i.e $M$ is relative Cohen-Macaulay with respect to $\fa$.}
\item[(ii)]{$\H^d_{\fm}(\H^{c}_{\fa}({M}))\cong \E_{R}(k)^{\tiny{r(M)}}$ and $\H^i_{\fm}(\H^{c}_{\fa}({M}))=0$ for all $i\neq d$. }
\item[(iii)]{$\Ext_{R}^{d}(k,\H^{c}_{\fa}(M))\cong k^{\tiny{r(M)}}$ and $\Ext_{R}^{i}(k,\H^{c}_{\fa}(M))=0$ for all $i\neq d$}.
\item[(iv)]{$\mu^i(\fm, \H_{\fa}^c(M))={\tiny{r(M)}}\delta_{di}$}.}
\end{itemize}
\emph{Moreover, if one of the above statements holds, then $\Ext^{i}_{R}(\H^{c}_\fa(M),\H^{c}_\fa(M))=0$ for all $i\neq0$ and  $\Hom_{R}(\H^{c}_\fa(M),\H^{c}_\fa(M))\cong\Hom_{R}(M,M)\otimes_{R}\hat{R}^{\fa}$, where $\hat{R}^{\fa}$ denotes the $\fa$-adic completion of $R$.}

\begin{proof}(i)$\Rightarrow$(ii): First, we can use \cite[Proposition 2.8]{MZ} and the assumption to see that $\H^i_{\fm}(\H^{c}_{\fa}({M}))\cong\H^{i+c}_{\fm}(M)$ for all $i\geq0$. Now, since $M$ is Cohen-Macaulay of dimension $n$, $\H^d_{\fm}(\H^{c}_{\fa}({M}))\cong\H^{n}_{\fm}(M)$ and $\H^i_{\fm}(\H^{c}_{\fa}({M}))=0$ for all $i\neq d$. On the other hand, in view of \cite[Theorem 2.5]{MZ} and the assumption, $\H^{n}_{\fm}(M)$ is an injective $R$-module. Therefore, one can use \cite[Corollay 2.2]{MZ} to achieve the isomorphism $\H^{n}_{\fm}(M)\cong{\E_{R}(k)}^{\tiny{r(M)}}.$

(ii)$\Rightarrow$(iii): It follows from Lemma 2.7 that $\Ext_{R}^{i}(k,\H^{c}_{\fa}(M))=0$ for all $i<d$. On the other hand, since $\H^d_{\fm}(\H^{c}_{\fa}({M}))$ is injective, one can use \cite[Proposition 2.1]{MZ} to see that $\Ext_{R}^{i}(k,\H^{c}_{\fa}(M))=0$ for all $d<i$ and that $\Ext_{R}^{d}(k,\H^{c}_{\fa}(M))\cong\Hom_{R}(k,\H^d_{\fm}(\H^{c}_{\fa}({M})))\cong k^{\tiny{r(M)}}$, and hence the assertion is done. The implications (iii)$\Leftrightarrow$(iv) is clear.

(iii)$\Rightarrow$(ii): First, in view of Lemma 2.7, $\H^i_{\fm}(\H^{c}_{\fa}({M}))=0$ for all $i<d$. Now, since $\dim_{R}\H^{c}_{\fa}({M})\leq\dim_{R}M/\fa M=d$, one can use the vanishing theorem to see that $\H^i_{\fm}(\H^{c}_{\fa}({M}))=0$ for all $d<i$. Therefore, \cite[Proposition 2.1]{MZ} implies that $\Hom_{R}(R/\fm, \H^d_{\fm}(\H^{c}_{\fa}({M})))\cong k^{\tiny{r(M)}}$, and so it is an Artinian $R$-module. Hence, by \cite[Theorem 7.1.2]{BRS}, $\H^d_{\fm}(\H^{c}_{\fa}({M}))$ is Artinian. Thus, one can use \cite[Corollary 2.2]{MZ} to see that $\mu^1(\fm,\H^d_{\fm}(\H^{c}_{\fa}({M})))=\mu^{d+1}(\fm,\H^{c}_{\fa}({M}))=0$ and $\mu^0(\fm,\H^d_{\fm}(\H^{c}_{\fa}({M})))=r(M)$. Hence $\H^d_{\fm}(\H^{c}_{\fa}({M}))$ is an injective $R$-module, and so we obtain the isomorphism $\H^d_{\fm}(\H^{c}_{\fa}({M}))\cong \E_{R}(k)^{\tiny{r(M)}}$.

(ii)$\Rightarrow$(i): Since $M$ is relative Cohen-Macaulay in $\Supp_{R}(M/\fa M)\setminus\{\fm\}$, one can use \cite[Proposition 2.8]{MZ} to deduce that $\H_{\fp R_{\fp}}^i(\H_{\fa R_{\fp}}^c(M_{\fp}))\cong \H_{\fp R_{\fp}}^{i+c}(M_{\fp})$ for all prime ideals $\fp$ in $\Supp_{R}(M/\fa M)\setminus\{\fm\}$ and for all $i\in\mathbb{Z}$. On the other hand, since $M$ is Gorenstein, one can use \cite[Theorem 2.5]{MZ} and \cite[Corollay 2.2]{MZ} to see that $\H_{\fm}^{n}({M})\cong\E_{R}(k)^{\tiny{r(M)}}$. Therefore, the assertion follows from \cite[Theorem 4.4]{HS}. Next, one can use Proposition 2.6(iii) and \cite[Proposition 2.1]{MZ} to establish the final assertion.

\end{proof}
\end{thm}

Next, we provide an example to show that even if, for all $\fp\in\Supp_{R}(M/\fa M)\setminus\{\fm\}$, $M_{\fp}$ is relative Cohen-Macaulay with respect to $\fa R_{\fp}$, but with different cohomological dimensions, then Theorem 2.9 is no longer true.
\begin{exam}\emph{Let $k$ be a filed and $R=k[[x,y,z]]$. Set $\fa:=(xy,xz)$. Then, we have $\gr(\fa,R)=\h\fa=1$. Also, one can use Mayer-Vietoris sequence to see that $\cd(\fa,R)=2$ and $\H_{\fa}^1(R)\cong\H_{(x)}^1(R)$. On the other hand, it is easy to see that, for all prime ideals $\fp$ in $\V(\fa)\setminus\{\fm\}$, $\fa R_{\fp}=(x)R_{\fp}$ or $\fa R_{\fp}=(y,z)R_{\fp}$; and hence $R_{\fp}$ is relative Cohen-Macaulay with respect to $\fa R_{\fp}$ for all prime ideals $\fp$ in $\V(\fa)\setminus\{\fm\}$ with $\cd(\fa R_{\fp},R_{\fp})\in\{1,2\}$. Next, since $\H_{(x)}^i(R)=0$ for all $i\neq1$, one can see that $\H_{\fm}^i(\H_{(x)}^1(R))\cong\H_{\fm}^{i+1}(R)$ for all $i$. Hence, by using the fact that $R$ is Gorenstein, we have $\H^d_{\fm}(\H^{1}_{\fa}({R}))\cong \E_{R}(k)$ and $\H^i_{\fm}(\H^{1}_{\fa}({R}))=0$ for all $i\neq d$, where $d=\dim R/\fa$. }
\end{exam}

Next, we single out a certain case of the above theorem for $M=\omega_{R}$, where $\omega_{R}$ denotes a canonical module of a Cohen-Macaulay ring $R$.
\begin{cor}Let $(R,\fm, k)$ be a local Cohen-Macaulay ring which admits a canonical $R$-module $\omega_{R}$ and let $\fa$ be an ideal of $R$ with \emph{$\h_{R}\fa=c$} and \emph{$\dim R/ \fa =d$}. Suppose that $\omega_{R}$ is relative Cohen-Macaulay in \emph{$\V(\fa)\setminus\{\fm\}$}. Then the following statements are equivalent.
\begin{itemize}
\item[(i)]{\emph{$\H^{i}_\fa(\omega_{R})=0$ for all $i\neq c$.}}
\item[(ii)]{\emph{$\H^d_{\fm}(\H^{c}_{\fa}({\omega_{R}}))\cong \E_{R}(k)$ and $\H^i_{\fm}(\H^{c}_{\fa}({\omega_{R}}))=0$ for all $i\neq d$. }}
\item[(iii)]{\emph{$\Ext_{R}^{d}(k,\H^{c}_{\fa}(\omega_{R}))\cong k$ and $\Ext_{R}^{i}(k,\H^{c}_{\fa}(\omega_{R}))=0$ for all $i\neq d$}.}
\item[(iv)]{\emph{$\mu^i(\fm,\H_{\fa}^c(\omega_{R}))=\delta_{di}$}}.
\end{itemize}
\emph{Moreover, if $\fa$ satisfies the above conditions, then $\Ext^{i}_{R}(\H^{c}_\fa(\omega_{R}),\H^{c}_\fa(\omega_{R}))=0$ for all $i\neq0$ and $\hat{R}^{\fa}\cong\Hom_{R}(\H^{c}_\fa(\omega_{R}),\H^{c}_\fa(\omega_{R}))$, where $\hat{R}^{\fa}$ denotes the $\fa$-adic completion of $R$.}
\end{cor}
\begin{proof} We first notice that $\Supp(\omega_{R}/\fa \omega_{R})=\Supp(\omega_R)\cap\V(\fa)=\V(\fa)$. Therefore, $\dim(\omega_{R}/\fa \omega_{R})=\dim( R/\fa)$ and  $\h_{R}\fa=\h_{\omega_{R}}\fa$. On the other hand, one can use \cite[Proposition 3.3.11]{BH} to see that $\omega_{R}$ is a Gorenstein $R$-module of type 1 and that $\Hom_{R}(\omega_{R},\omega_{R})\cong R$. Hence, the assertion follows from Theorem 2.9.

\end{proof}

The following corollary, which is an immediate consequence of Corollary 2.11, has been proved in \cite[Theorem 0.1]{HS1} as a main result.

\begin{cor} Let $(R,\fm, k)$ be a local Gorenstein ring and let $\fa$ be an ideal of $R$. Set \emph{$\h_{R}\fa=c$} and \emph{$\dim R/ \fa =d$}. Suppose that $R$ is relative Cohen-Macaulay in \emph{$\V(\fa)\setminus\{\fm\}$}. Then the following statements are equivalent.
\begin{itemize}
\item[(i)]{\emph{$\H^{i}_\fa(R)=0$ for all $i\neq c$, i.e $R$ is relative Cohen-Macaulay with respect to $\fa$.}}
\item[(ii)]{\emph{$\H^d_{\fm}(\H^{c}_{\fa}({R}))\cong \E_{R}(k)$ and $\H^i_{\fm}(\H^{c}_{\fa}({R}))=0$ for $i\neq d$ }.}
\item[(iii)]{\emph{$\Ext_{R}^{d}(k,\H^{c}_{\fa}(R))\cong k$ and $\Ext_{R}^{i}(k,\H^{c}_{\fa}(R))=0$ for all $i\neq d$}.}
\item[(iv)]{\emph{$\mu^i(\fm,\H_{\fa}^c(R))=\delta_{di}$}}.
\end{itemize}
\emph{Moreover, if one of the above conditions is satisfied, then $\hat{R}^{\fa}\cong\Hom_{R}(\H^{c}_\fa(R),\H^{c}_\fa(R))$ and $\Ext^{i}_{R}(\H^{c}_\fa(R),\H^{c}_\fa(R))=0$ for all $i\neq0$.}
\end{cor}

{$\mathbf{Acknowledgements}.$} I am very grateful to Professor Hossein Zakeri for his kind comments and assistance in the preparation of this paper. Also, the author is grateful to the referee for careful reading and for suggesting several improvements of the manuscript.



\end{document}